\documentclass[12pt]{amsart}
\usepackage{amscd,amsmath,amssymb,amsthm,amsfonts,epsfig,graphics}
\usepackage{geometry}
\usepackage{graphicx}
\usepackage{mathrsfs,amssymb}
\usepackage{bm}
\theoremstyle{plain}

\newtheorem{lemma}{Lemma}

\newtheorem{proposition}{Proposition}
\newtheorem{remark}{Remark}

\newtheorem{theorem}{Theorem}
\numberwithin{equation}{section}

\begin{document}
\title[IMPROVED CRITICAL EIGENFUNCTION RESTRICTION ESTIMATES IN 2-D]{IMPROVED CRITICAL EIGENFUNCTION RESTRICTION ESTIMATES ON RIEMANNIAN SURFACES WITH NONPOSITIVE CURVATURE}
\author{Yakun Xi}
\author{Cheng Zhang}
\address{Department of Mathematics\\
Johns Hopkins University\\
Baltimore, MD 21218, USA}
\email{ykxi@math.jhu.edu,
czhang67@math.jhu.edu}
\date{}

\begin{abstract}We show that one can obtain improved $L^4$ geodesic restriction estimates for eigenfunctions on compact Riemannian surfaces with nonpositive curvature. We achieve this by adapting Sogge's strategy in \cite{loglog}. We first combine the improved $L^2$ restriction estimate of Blair and Sogge \cite{top} and the classical improved $L^\infty$ estimate of B\'erard to obtain an improved weak-type $L^4$ restriction estimate. We then upgrade this weak estimate to a strong one by using the improved Lorentz space estimate of Bak and Seeger \cite{seeger}.  This estimate improves the $L^4$ restriction estimate of Burq, G\'erard and Tzvetkov \cite{burq} and Hu \cite{hu} by a power of $(\log\log\lambda)^{-1}$. Moreover, in the case of compact hyperbolic surfaces, we obtain further improvements in terms of $(\log\lambda)^{-1}$ by applying the ideas from \cite{chensogge} and \cite{top}.  We are able to compute various constants that appeared in \cite{chensogge} explicitly, by proving detailed oscillatory integral estimates and lifting calculations to the universal cover $\mathbb H^2$.
\end{abstract}

\maketitle

\section{Introduction}
Let $(M,g)$ be a compact $n$-dimensional Riemannian manifold and let $\Delta_g$ be the associated Laplace-Beltrami operator. Let $e_\lambda$ denote the $L^2$-normalized eigenfunction 
\[-\Delta_g e_\lambda=\lambda^2e_\lambda,\]
so that $\lambda$ is the eigenvalue of the first order operator $\sqrt{-\Delta_g}.$

Various types of concentrations exhibited by eigenfunctions have been studied. See the recent survey by Sogge \cite{survey} for a detailed discussion.  A classical result of Sogge \cite{Slp} states that the $L^p$ norm of the eigenfunctions satisfies
\[\|e_\lambda\|_{L^p{(M)}}\le C\lambda^{\mu{(p)}},\]
where $2\le p\le\infty$ and $\mu(p)$ is given by
\[\mu(p)=\max\bigg\{\frac{n-1}{2}\bigg(\frac{1}{2}-\frac{1}{p}\bigg),\ n\bigg(\frac{1}{2}-\frac{1}{p}\bigg)-\frac{1}{2}\bigg\}.\]
If we let $p_c=\frac{2(n+1)}{n-1}$, these bounds can also be written as
\begin{eqnarray}\label{slp}\|e_\lambda\|_{L^p{(M)}}\le
\begin{cases}
C\lambda^{\frac{n-1}{2}(\frac{1}{2}-\frac{1}{p})}, &2\le p\le p_c,
\cr C\lambda^{n(\frac{1}{2}-\frac{1}{p})-\frac{1}{2}}, &p_c\le p\le\infty. \end{cases}
\end{eqnarray}
The estimates \eqref{slp} are saturated on the round sphere by zonal functions for $p\ge p_c$ and for $2<p\le p_c$ by the highest weight spherical harmonics. Even though they are sharp on the round sphere $S^n$, it is expected that \eqref{slp} can be improved for generic Riemannian manifolds. Manifolds with nonpositive sectional curvature have been studied as the model case for such improvements.

It is well-known that one can get log improvements for $\|e_\lambda\|_{L^\infty(M)}$ if $M$ has nonpositive  curvature. Indeed, B\'erard's results \cite{berard} in 1977 on improved error term estimates for the Weyl formula imply that
\[\|e_\lambda\|_{L^\infty(M)}=O(\lambda^\frac{n-1}2/\sqrt{\log\lambda}),\]
which gives log improvements over \eqref{slp} for $p>p_c$ via interpolation. Recently, Blair and Sogge \cite{top} were able to obtain log improvements over \eqref{slp} for $2<p<p_c$ by proving improved Kakeya-Nikodym bounds which measure $L^2$-concentration of eigenfunctions on $\lambda^{-\frac12}$ tubes about unit length geodesics. Despite the success in improving \eqref{slp} for the range $2<p<p_c$ and $p_c<p<\infty$, improvement of the critical case has been elusive. On one hand, the estimate at the critical exponent $p_c=\frac{2(n+1)}{n-1}$
\begin{equation}\label{pc}
\|e_\lambda\|_{L^\frac{2(n+1)}{n-1}(M)}\le C\lambda^\frac{n-1}{2(n+1)}
\end{equation}
actually implies \eqref{slp} for all $2\le p\le\infty$ via interpolating with the classical $L^\infty$ estimate and the trivial $L^2$ estimate. On the other hand, this bound \eqref{pc} is sensitive to both point concentration and concentration along  geodesics, in the sense that it is saturated by both zonal functions and spherical harmonics on the round sphere. 

Recently, Sogge \cite{loglog} managed to improve over \eqref{pc} by a power of $(\log\log\lambda)^{-1}$ under the assumption of nonpositive curvature. Using Bourgain's \cite{bourgain} idea in proving weak-type estimate for the Stein-Tomas restriction theorem, Sogge was able to combine the recent improved $L^p,$ $2<p<p_c$ bounds of Blair and Sogge \cite{top} and the classical improved sup-norm estimate of B\'erard \cite{berard}, to get improved bounds for the critical case.

In the last decade, similar $L^p$ estimates have been established for the restriction of eigenfunctions to geodesics. Burq, G\'erard and Tzvetkov \cite{burq} and Hu \cite{hu} showed that for $n$-dimensional Riemannian manifold $(M,g)$, if $\it\Pi$ denotes the space of all unit-length geodesics $\gamma$, then 

\begin{equation}\label{rest}\sup\limits_{\gamma\in\it\Pi}\Big(\int_\gamma|e_\lambda|^p\,ds\Big)^\frac1p\le C\lambda^{\sigma(n,p)}\|e_\lambda\|_{L^2(M)},
\end{equation}
where
\begin{eqnarray}\sigma(2,p)=
\begin{cases}
\frac14, &2\le p\le 4,
\cr \frac12-\frac1p, &4\le p\le\infty. \end{cases}
\end{eqnarray}
and
\begin{equation}\sigma(n,p)=\frac{n-1}2-\frac1p,\ {\rm if}\ {p}\ge2\ {\rm and}\ n\ge3,
\end{equation}
here the case $n=3,\ p=2$ is due to Chen and Sogge \cite{chensogge}. Note that in the 2-dimensional case, the estimates \eqref{rest} have a similar flavor compared to Sogge's $L^p$ estimates \eqref{slp}. Indeed, when $n=2$ the estimates \eqref{rest} also have a critical exponent $p_c=4$. Moreover, on the sphere $S^2$, \eqref{rest} is saturated by zonal functions when $p\le4$, while for $p\ge4$, it is saturated by the highest weight spherical harmonics. When $n=3$, the critical exponent no longer appears in \eqref{rest}. However, the estimate for  $p=2$ is still saturated by both zonal functions and highest weight spherical harmonics. In higher dimensions $n>3$, geodesic restriction estimates are too singular to detect concentrations of eigenfunctions near geodesics. In fact, in these dimensions, estimates \eqref{rest} are always saturated by zonal functions rather than highest weight spherical harmonics on the round sphere $S^n$. 

There has been considerable work towards improving \eqref{rest} under the assumption of nonpositive curvature in the 2-dimensional case. B\'erard's sup-norm estimate \cite{berard} provides natural improvements for large $p$. In \cite{chen}, Chen managed to improve over \eqref{rest} for all $p>4$ by a $(\log\lambda)^{-\frac12}$ factor:
\begin{equation}
\sup\limits_{\gamma\in\it\Pi}\Big(\int_\gamma|e_\lambda|^p\,ds\Big)^\frac1p\le C\frac{\lambda^{\frac12-\frac1p}}{(\log\lambda)^\frac12}\|e_\lambda\|_{L^2(M)}.
\end{equation}
Sogge and Zelditch \cite{sz} showed that one can improve \eqref{rest} for $2\le p<4$, in the sense that
\begin{equation}
\sup\limits_{\gamma\in\it\Pi}\Big(\int_\gamma|e_\lambda|^p\,ds\Big)^\frac1p=o(\lambda^\frac14).
\end{equation}
A few years later, Chen and Sogge \cite{chensogge} showed that the same conclusion can be drawn for $p=4$:
\begin{equation}\label{chensogge}
\sup\limits_{\gamma\in\it\Pi}\Big(\int_\gamma|e_\lambda|^4\,ds\Big)^\frac14=o(\lambda^\frac14).
\end{equation}
\eqref{chensogge} is the first result to improve an estimate that is saturated both by zonal functions and highest weight spherical harmonics.
Recently, by using the Toponogov's comparison theorem, Blair and Sogge \cite{top} showed that it is possible to get log improvements for $L^2$-restriction:
\begin{equation}\label{top1}
\sup\limits_{\gamma\in\it\Pi}\Big(\int_\gamma|e_\lambda|^2\,ds\Big)^\frac12\le C\frac{\lambda^\frac14}{(\log\lambda)^\frac14}\|e_\lambda\|_{L^2(M)},
\end{equation}

Adapting Sogge's idea in proving improved critical $L^{p_c}$ estimates \cite{loglog}, we are able to further improve the critical $L^4$-restriction estimate in the 2-dimensional case \eqref{rest} by a factor of $(\log\log\lambda)^{-\frac18}.$
\begin{theorem}\label{mainthm}
Let $(M,g)$ be a 2-dimensional compact Riemannian manifold of nonpositive curvature, let $\gamma\subset M$ be a fixed unit-length geodesic segment. Then for $\lambda\gg 1$, there is a constant $C$ such that
\begin{equation}\label{main}
\|\chi_{[\lambda,\lambda+(\log\lambda)^{-1}]}f\|_{L^4(\gamma)}\le C\lambda^\frac14(\log\log\lambda)^{-\frac18}\|f\|_{L^2(M)}.\end{equation}
Therefore, taking $f=e_\lambda$, we have
\begin{equation}\|e_\lambda\|_{L^4(\gamma)}\le C\lambda^\frac14(\log\log\lambda)^{-\frac18}\|e_\lambda\|_{L^2(M)}.\end{equation}
Moreover, if $\it\Pi$ denotes the set of unit-length geodesics, there exists a uniform constant $C=C(M,g)$ such that
\begin{equation}\label{uniform}\sup\limits_{\gamma\in\it\Pi}\Big(\int_\gamma|e_\lambda|^4\,ds\Big)^\frac14\le C\lambda^\frac14(\log\log\lambda)^{-\frac18}\|e_\lambda\|_{L^2(M)}.
\end{equation}
\end{theorem}

Furthermore, if we assume further that $M$ has constant negative curvature, we are able to get log improvement for the $L^4$-restriction estimate following the ideas in \cite{top} and \cite{chensogge}.
\begin{theorem}\label{thm2}
Let $(M,g)$ be a 2-dimensional compact Riemannian manifold of constant negative curvature, let $\gamma\subset M$ be a fixed unit-length geodesic segment. Then for $\lambda\gg 1$, there is a constant $C$ such that
\begin{equation}\|e_\lambda\|_{L^4(\gamma)}\le C\lambda^\frac14(\log\lambda)^{-\frac14}\|e_\lambda\|_{L^2(M)}.\end{equation}
Moreover, if $\it\Pi$ denotes the set of unit-length geodesics, there exists a uniform constant $C=C(M,g)$ such that
\begin{equation}\sup\limits_{\gamma\in\it\Pi}\Big(\int_\gamma|e_\lambda|^4\,ds\Big)^\frac14\le C\lambda^\frac14(\log\lambda)^{-\frac14}\|e_\lambda\|_{L^2(M)}.
\end{equation}
\end{theorem}

Our paper is organized as follows. In Section 2, we give the proof of Theorem \ref{mainthm}. We do this by first proving a new local restriction estimate which corresponds to Lemma 2.2 in \cite{loglog}. Then we use this local estimate together with the improved $L^2$-restriction estimate \eqref{top1} of Blair and Sogge \cite{top} and the classical improved sup-norm estimate of B\'erard \cite{berard} to obtain improved $L^2(M)\rightarrow L^{4,\infty}(\gamma)$ estimate. Finally, we prove Theorem \ref{mainthm} by interpolating between the improved $L^2(M)\rightarrow L^{4,\infty}(\gamma)$ estimate and the $L^2(M)\rightarrow L^{4,2}(\gamma)$ estimate of Bak and Seeger \cite{seeger}. In Section 3, we show how to obtain further improvements under the assumption of constant negative curvature. We  follow the strategies that were introduced in \cite{chensogge} and \cite{top}. We shall lift all the calculations to the universal cover $\mathbb H^2$ and then use the Poincar\'e half-plane model to compute the dependence of various constants explicitly. 

Throughout our argument, we shall assume that the injectivity radius of $M$ is sufficiently large, and fix $\gamma$ to be a unit length geodesic segment. We shall use $P$ to denote the first order operator $\sqrt{-\Delta_g}$. Also, whenever we write $A\lesssim B$, it means $A\le CB$ and $C$ is some unimportant constant.
\section{Riemannian surface with nonpositive curvature}

We start with some standard reductions. Let $\rho\in S(\mathbb R)$ such that $\rho(0)=1$ and ${\rm supp}\ \hat\rho\subset[-1/2,1/2]$, then it is clear that the operator $\rho(T(\lambda-P))$ reproduces eigenfunctions, in the sense that
\[\rho(T(\lambda-P))e_\lambda=e_\lambda.\]
Consequently, we would have the estimate \eqref{main} if we could show that 
\begin{equation}\label{strong}
\|\rho(\log\lambda(\lambda-P))\|_{L^2(M)\rightarrow L^{4}(\gamma)}=O(\lambda^\frac14/(\log\log\lambda)^\frac18).
\end{equation}
The uniform bound \eqref{uniform} also follows by a standard compactness argument.
\subsection{A local restriction estimate}
To prove \eqref{strong}, we apply Sogge's strategy in \cite{loglog}. We shall need the following local restriction estimate.
\begin{lemma} \label{lemma1}Let $\lambda^{-1}\le r\le1$, $\gamma_r$ be a fixed subsegment of $\gamma$ with length $r$. Then we have
\[\|\rho(\lambda-P)f\|_{L^2(\gamma_r)}\lesssim \lambda^{\frac14}r^\frac14\|f\|_{L^2(M)}.\]
\end{lemma}
\begin{proof}
By a standard $TT^*$ argument, this is equivalent to showing that
\begin{equation}\label{TT}\|\chi(\lambda-P)h\|_{L^2(\gamma_r)}\lesssim \lambda^{\frac12}r^\frac12\|h\|_{L^2(\gamma_r)},\end{equation}
here $\chi=|\rho|^2$. Thus 
\[\chi(\lambda-P)h=\frac1{2\pi}\int\widehat\chi(t)e^{-i\lambda t}e^{itP}h\,dt.\]
We shall need a preliminary reduction. Let $\beta\in C_0^\infty$ be a Littlewood-Paley bump function, satisfying
\[\beta(s)=1,\ \mathrm{if}\ s\in[1/2,2],\ \ \mathrm{and}\ \beta(s)=0,\ \mathrm{if}\ s\not\in[1/4,4].\]
Then we claim that it suffices to prove:
\begin{equation}\label{LP}\Big\|\int\widehat\chi(t)e^{-i\lambda t}\beta(P/\lambda)e^{itP}h\,dt\Big\|_{L^2(\gamma)}\le C\lambda^\frac12 r^\frac12\|h\|_{L^2(\gamma)}.\end{equation}

Indeed, we note that the operator
\begin{equation}\label{R}\int\widehat\chi(t)e^{-i\lambda t}(1-\beta(P/\lambda))e^{itP}\,dt\end{equation}
has kernel
\[\sum\widehat\chi(\lambda-\lambda_j)(1-\beta)(\lambda_j/\lambda)e_j(\gamma(s))\overline{e_j(\gamma(s'))}.\]
Since $\chi\in C_0^\infty(\mathbb R)$ and $\beta$ is the Littlewood-Paley bump function, we see that
\[|\widehat\chi(\lambda-\lambda_j)(1-\beta)(\lambda_j/\lambda)|\le C(1+\lambda+\lambda_j)^{-4}.\]
On the other hand, by the Weyl formula,
\[\sum\limits_{\lambda_j\in[\lambda,\lambda+1]}|e_j(\gamma(s))e_j(\gamma(s'))|\le C(1+\lambda),\]
we conclude that the kernel of the operator given by \eqref{R} is $O(\lambda^{-1})$. This means that this operator enjoys better bounds than \eqref{TT}, which gives our claim that it suffices to prove \eqref{LP}.

To prove \eqref{LP}, we consider the corresponding kernel
\[K_\lambda(\gamma(s),\gamma(s'))=\int\widehat\chi(t)e^{-i\lambda t}\beta(P/\lambda)e^{itP}(\gamma(s),\gamma(s'))\,dt.\]
We claim that $K_\lambda$ satisfies
\begin{equation}
\label{kernel}
|K_\lambda(\gamma(s),\gamma(s'))|=O(\lambda^\frac12|s-s'|^{-\frac12}).
\end{equation}
Indeed, one may use a parametrix and the calculus of Fourier integral operators to see that modulo a trivial error term of size $O(\lambda^{-N})$
\[(\beta(P/\lambda)e^{itP})(\gamma(s),\gamma(s'))=\int_{\mathbb R^2}e^{i(s-s')\xi_1+it|\xi|}\,\alpha(t,s,s',|\xi|)\,d\xi,\]
where $\alpha$ is a zero-order symbol. See \cite{fio} and \cite{hangzhou}. Thus, modulo trivial errors,
\begin{equation}\label{kernel2}K_\lambda(\gamma(s),\gamma(s'))=\int\int_0^\infty e^{it(l-\lambda)}\,{\alpha(t,s,s',l)}\,\Big(\int_{S^1}e^{il(s-s')\langle(0,1),\omega\rangle}\,d\omega\Big)\,l\,dldt.\end{equation}
Integrating by parts in $t$ shows that the above expression is majorized by
\[\int_0^\infty(1+|l-\lambda|)^{-3}l\,dl=O(\lambda),\]
thus \eqref{kernel} is valid when $|s-s'|\le\lambda^{-1}$. To handle the remaining case, we recall that, by stationary phase,
\[\int_{S^1}e^{ix\cdot\omega}\,d\omega=O(|x|^{-\frac12}),\ \ \ |x|\ge1.\]
If we plug this into \eqref{kernel2} with $x=l(s-s',0)$, and integrate by parts in $t$, we conclude that if $\lambda^{-1}\le|s-s'|$, we have
\[|K_\lambda(\gamma(s),\gamma(s'))|\le\int_0^\infty(1+|l-\lambda|)^{-3}\,(l|s-s'|)^{-\frac12}\,l\,dl=O(\lambda^\frac12|s-s'|^{-\frac12}),\]
as claimed. By Young's inequality, the left hand side of \eqref{LP} is bounded by
\[\lambda^\frac12\Big(\int_0^r\Big|\int_0^r \frac 1{|s-s'|^\frac12}\,h(s')\,ds' \Big|^2\,ds\Big)^\frac12\le \lambda^\frac12r^\frac12\|h\|_{L^2([0,r])},\]
completing our proof.
\end{proof}
\begin {remark} \label{remark1}\rm{A similar argument gives the same estimate for $\rho(T(\lambda-P))f$ if $T\ge1$. Indeed, the same argument works for operator with kernel 
\begin{equation}\label{a}\Big[\int a(t)\,e^{it\lambda}\,e^{-itP}\,dt\Big](\gamma(s),\gamma(s')),\end{equation}
providing $a\in C_0^\infty(-1,1)$. While the operator $\rho(T(\lambda-P))$ corresponds to the kernel \[\Big[\frac 1 T\int a(t/T)\,e^{it\lambda}\,e^{-itP}\,dt\Big](\gamma(s),\gamma(s')),\] which is just an averaged version of \eqref{a}, thus it satisfies the same estimate.}
\end{remark}
\subsection{An improved weak-type estimate}
In this section, we prove the following improved weak-type estimate.
\begin{proposition}\label{weakpop}
Let $(M,g)$ be a 2-dimensional compact Riemannian manifold of nonpositive curvature. Then for $\lambda\gg 1$
\begin{equation}\label{weak}
\|\rho(\log\lambda(\lambda-P))\|_{L^2(M)\rightarrow L^{4,\infty}(\gamma)}=O(\lambda^\frac14/(\log\log\lambda)^\frac14).
\end{equation}
\end{proposition}
As discussed before, the $L^4$ restriction bound is saturated by both zonal functions and highest weight spherical harmonics. Thus as in \cite{loglog}, to get improved $L^4$ bounds, we shall need the following two improved results which corresponds to the range $2\le p<4$ and the range $4<p\le\infty$ respectively.
\begin{lemma}[\cite{top}] Let $(M,g)$ be as above. Then for $\lambda\gg1$ we have
\begin{equation}
\label{top}\|\rho(\log\lambda(\lambda-P))\|_{L^2(M)\rightarrow L^2(\gamma)}=O(\lambda^\frac1 4/(\log \lambda)^\frac14).\end{equation}
\end{lemma}

\begin{lemma}[\cite{berard}] If $(M,g)$ is as above then there is a constant $C=C(M,g)$ so that for $T\ge1$ and large $\lambda$ we have the following bounds for the kernel of $\eta(T(\lambda-P))$, $\eta=\rho^2$,
\begin{equation}\label{berard}
|\eta(T(\lambda-P))(x,y)|\le C\left[T^{-1}\Big(\frac{\lambda}{d_g(x,y)}\Big)^\frac12+\lambda^\frac12e^{CT}\right],
\end{equation}
\end{lemma}

The first Lemma is a recent result of Blair and Sogge \cite{top}. The other bound \eqref{berard} is well-known and follows from the arguments in the paper of B\'erard \cite{berard}.

Now we are ready to prove Proposition \ref{weakpop}. It suffices to show that 
\begin{equation}
|\{x\in\gamma:|\rho(\log\lambda(\lambda-P))f(x)|>\alpha\}|\le C\alpha^{-4}\lambda(\log\log\lambda)^{-1}.
\end{equation}
assuming $f$ is $L^2$ normalized.
By Chebyshev inequality and \eqref{top}, we have
\begin{align*}|\{x\in\gamma:|\rho(\log\lambda(\lambda-P))f(x)|>\alpha\}|&\le \alpha^{-2}\int_\gamma|\rho(\log\lambda(\lambda-P))f|^2\,ds\\
&\le\alpha^{-2}\lambda^\frac12(\log\lambda)^{-\frac12}.\end{align*}
Note that for large $\lambda$ we have
\[\alpha^{-2}\lambda^\frac12(\log\lambda)^{-\frac12}\ll \alpha^{-4}\lambda(\log\log\lambda)^{-1},\ \ \ \ \mathrm{if}\ \alpha\le \lambda^\frac14(\log\lambda)^\frac18.\]
Thus it remains to show
\[|\{x\in\gamma:|\rho(\log\lambda(\lambda-P))f(x)|>\alpha\}|\le C\alpha^{-4}\lambda(\log\log\lambda)^{-1},
\]
when $\alpha\ge\lambda^\frac14(\log\lambda)^\frac18$.

We notice that
\[\big|\big[\rho(c_0\log\log\lambda(\lambda-\tau))-1\big]\rho(\log\lambda(\lambda-\tau))\big|\lesssim\frac{\log\log\lambda}{\log\lambda}(1+|\lambda-\tau|)^{-N},\] 
together with the estimate
\[\|\chi_\lambda\|_{L^2(M)\rightarrow L^4(\gamma)}=O(\lambda^\frac14),\]
we see that
\[\big\|\big[\rho(c_0\log\log\lambda(\lambda-P))-I\big]\circ\rho(\log\lambda(\lambda-P))f\big\|_{L^4(\gamma)}\lesssim \frac{\log\log\lambda}{\log\lambda}\lambda^\frac14\|f\|_{L^2(M)}.\]
Therefore we would be done if we could show that
\[|\{x\in\gamma:|\rho(c_0\log\log\lambda(\lambda-P))h(x)|>\alpha\}|\le C\alpha^{-4}\lambda(\log\log\lambda)^{-1},
\]
if $\alpha\ge\lambda^\frac14(\log\lambda)^\frac18$, and $\|h\|_{L^2(M)}=1.$

Let 
$$A=\{x\in\gamma:|\rho(c_0\log\log\lambda(\lambda-P))h(x)|>\alpha\}.$$
Take \[r=\lambda\alpha^{-4}(\log\log\lambda)^{-2}.\] We decompose $A$ into $r$-separated subsets $\cup_j A_j=A$ with length $\approx r$. By replacing $A$ by a set of proportional measure, we may assume that if $j\neq k$, we have $\mathrm{dist}(A_j,A_k)>C_0r$, where $C_0$ will be specified momentarily.

Let $T_\lambda=\rho(c_0\log\log\lambda(\lambda-P))$, which has dual operator $T_\lambda^*$ mapping $L^2(\gamma)\rightarrow L^2(M)$. Let  $\psi_\lambda(x)=T_\lambda f(x)/|T_\lambda f(x)|$, if $T_\lambda f(x)\neq 0$, otherwise let $\psi_\lambda(x)=1$. Let $S_\lambda=T_\lambda T^*_\lambda$ and $a_j=\overline{\psi_\lambda\mathrm {1}_{A_j}}$. Then by Chebyshev's inequality and Cauchy-Schwarz inequality, we have
\begin{align*}\alpha|A|&\le\left|\int_\gamma T_\lambda f\overline{\psi_\lambda\mathrm{1}_A}\,ds\right|
\le\left|\int_\gamma\sum\limits_j T_\lambda f a_j \,ds\right|\\
&=\left|\int_M\sum\limits_j T^*_\lambda a_j f\,dV_g\right|\le\left(\int_M\Big|\sum\limits_j T^*_\lambda a_j\Big|^2\,dV_g\right)^\frac12,
\end{align*}
squaring both sides, we see that
\[\alpha^2|A|^2\le\sum\limits_j\int_M|T^*_\lambda a_j|^2dV_g+\sum\limits_{j\neq k}\int_\gamma S_\lambda a_j \overline{a_k}\,ds=I+II.\]
By the dual version of Lemma \ref{lemma1} (see Remark \ref{remark1}), we see that 
\[I\lesssim r^\frac12\lambda^\frac12 \sum\limits_j\int_\gamma|a_j|^2\,ds=r^\frac12\lambda^\frac12|A|=\lambda\alpha^{-2}(\log\log\lambda)^{-1}|A|.\]
By making $c_0$ sufficiently small, we see from \eqref{berard} that we can control the kernel, $K_\lambda(s,s')$, of $S_\lambda$ by
\[|K_\lambda(s,s')|\le C\Big[(\log\log\lambda)^{-1}\Big(\frac{\lambda}{|s-s'|}\Big)^\frac12+\lambda^\frac12(\log\lambda)^\frac1{40}\Big],\]
thus
\begin{align*}II &\lesssim\Big[(\log\log\lambda)^{-1}\Big(\frac{\lambda}{C_0 r}\Big)^\frac12+\lambda^\frac12(\log\lambda)^\frac1{40}\Big]\sum\limits_{j\neq k}\|a_j\|_{L^1}\|a_k\|_{L^1}\\
&\le C_0^{-\frac12}
\alpha^2|A|^2+\lambda^\frac12 (\log\lambda)^\frac1 {40}|A|^2.\end{align*}
Since we are assuming $\alpha\ge\lambda^\frac14(\log\lambda)^\frac18$, for sufficiently large $C_0$, we have \[II\le \frac12 \alpha^2|A|^2,\]
thus
\[\alpha^2|A|^2\le C\lambda\alpha^{-2}(\log\log\lambda)^{-1}|A|+ \frac12 \alpha^2|A|^2,\]
which gives
\[|A|\le C\lambda\alpha^{-4}(\log\log\lambda)^{-1},\ \ \ \rm{if}\ \alpha\ge\lambda^\frac14(\log\lambda)^\frac18,\]
completing the proof of Proposition \ref{weakpop}.

\subsection{Proof of Theorem \ref{mainthm}}
We shall combine the improved $L^2(M)\rightarrow L^{4,\infty}(\gamma)$ estimate \eqref{weak} we obtained in the last section with the following $L^2(M)\rightarrow L^{4,2}(\gamma)$ estimate established by Bak and Seeger \cite{seeger} to prove our main theorem. This estimate of Bak and Seeger holds for general Riemannian manifold without any curvature condition.
\begin{lemma}[\cite{seeger}] \label{lemma2}Let $(M,g)$ be a 2-dimensional Riemannian manifold. Fix $\gamma\subset M$ to be a geodesic segment. Then we have the following estimate for the unit band spectral projection operator $\chi_{[\lambda,\lambda+1]}$
\begin{equation}\label{bak}
\|\chi_{[\lambda,\lambda+1]}f\|_{L^{4,2}(\gamma)}\le C(1+\lambda)^\frac14\|f\|_{L^2(M)}.
\end{equation}r
\end{lemma}
We remark that Lemma \ref{lemma2} is a special case of the results in \cite{seeger} regarding the restriction of eigenfunctions to hypersurfaces for manifolds with dimension $n\ge2$.

Let us recall some basic facts about the Lorentz space $L^{p,q}(\gamma)$. First,  for a function $u$ on $M$, we define the corresponding distribution function $\omega(\alpha)$ with respect to $\gamma$ as
\[\omega(\alpha)=|\{x\in\gamma:|u(x)|>\alpha\},\ \ \ \alpha>0. \]  $u^*$ is the nonincreasing rearrangement of $u$ on $\gamma$, given by
\[u^*(t)=\inf\{\alpha:\omega(\alpha)\le t\},\ \ \ t>0.\]
Then the Lorentz space $L^{p,q}(\gamma)$ for $1\le p<\infty$ and $1\le q<\infty$ is defined as all $u$ so that
\begin{equation}
\|u\|_{L^{p,q}(\gamma)}=\Big(\frac q p\int_0^\infty\big[t^\frac1pu^*(t)\big]^q\frac{dt}t\Big)^\frac1q<\infty,
\end{equation}

It's well known that for the special case $p=q$, the Lorentz norm $\|{}\cdot{}\|_{L^{p,p}(\gamma)}$ agrees with the standard $L^p$ norm $\|{}\cdot{}\|_{L^p(\gamma)}$. Moreover, we also have
\[\sup\limits_{t>0}t^\frac1pu^*(t)=\sup\limits_{\alpha>0}\alpha[\omega(\alpha)]^\frac1p.\]

If we take $u=\chi_{[\lambda,\lambda+(\log\lambda)^{-1}]}f$, and assume $\|f\|_{L^2(M)}=1$, then by \eqref{weak} we have

\begin{equation}\label{weak'}
\sup\limits_{t>0}t^\frac14 u^*(t)\le C\lambda^\frac14 (\log\log\lambda)^{-\frac14}.
\end{equation}

On the other hand, since $\chi_{[\lambda,\lambda+1]}u=u$, by Lemma \ref{lemma2} we have
\begin{equation}
\label{4,2}
\|u\|_{L^{4,2}(\gamma)}\le C\lambda^\frac14\|u\|_{L^2(M)}\le C\lambda^\frac14.\end{equation}

Interpolating between \eqref{weak'} and \eqref{4,2}, we then get
\begin{align*}
\|u\|_{L^4(\gamma)}&=\left(\int_0^\infty\big[t^\frac14u^*(t)\big]^4\frac{dt}t\right)^\frac14\\
&\lesssim \left(\sup\limits_{t>0}t^\frac14u^*(t)\right)^\frac12\|u\|_{L^{4,2}(\gamma)}^\frac12\\
&\lesssim \big(\lambda^\frac14(\log\log\lambda)^{-\frac14}\big)^\frac12\lambda^\frac18\\
&=\lambda^\frac14(\log\log\lambda)^{-\frac18},
\end{align*}
which completes the proof of Theorem \ref{mainthm}.

\section{Riemannian surfaces with constant negative curvature}
We shall apply the strategies in \cite{chensogge} and \cite{top} to prove Theorem \ref{thm2}. Recall in \cite{chensogge}, Chen and Sogge showed that for Riemannian surfaces with nonpositive curvature,
\begin{equation}\label{cmpmain}\begin{gathered}
  {\left( {\int_{0}^{1} {{{\left| {\int_{0}^{1} {\chi (T(\lambda  - P))(\gamma (t),\gamma (s))h(s)ds} } \right|}^4}dt} } \right)^{\frac{1}
{4}}} \hfill \\
   \le C{T^{ - \frac{1}
{2}}}{\lambda ^{\frac{1}
{2}}}{\left\| h \right\|_{{L^{\frac{4}
{3}}}([0,1])}} + {C_T}{\lambda ^{\frac{3}
{8}}}{\left\| h \right\|_{{L^{\frac{4}
{3}}}([ 0,1])}},\hfill \\
\end{gathered}
\end{equation}
here $\chi (T(\lambda  - P))(x,y)$ denotes the kernel of the multiplier operator $\chi (T(\lambda  - P))$. Clearly, this would imply \eqref{chensogge} if one takes $T$ to be sufficiently large. We shall show that under the assumption of constant negative curvature, the constant $C_T$ in \eqref{cmpmain} can be taken to be $e^{CT}$ where $C>0$ is some constant independent of $T$. Then Theorem \ref{thm2} would be proved if we set $T=c\log\lambda$, for $c>0$ to be sufficiently small. From now on, we shall use $C$ to denote various positive constants that are independent of $T$.

\subsection{Some reductions}
Choose a bump function $\beta \in C_0^\infty(\mathbb{R})$ satisfying
\[\beta(\tau)=1 \quad {\rm for} \quad |\tau|\le3/2,\quad  {\rm and} \quad \beta(\tau)=0, \quad |\tau|\ge2.\]
Then we may write
\[\begin{gathered}
  \chi (T(\lambda  - P))(x,y) = \frac{1}
{{2\pi T}}\int {\beta (\tau )\hat \chi (\tau /T){e^{i\lambda \tau }}({e^{ - i\tau P}})(x,y)d\tau }  \hfill \\
   + \frac{1}
{{2\pi T}}\int {(1 - \beta (\tau ))\hat \chi (\tau /T){e^{i\lambda \tau }}({e^{ - i\tau P}})(x,y)d\tau }  = {K_0}(x,y) + {K_1}(x,y). \hfill \\
\end{gathered}
\]
As \eqref{kernel} in the proof of Lemma 1, one may use the Hadamard parametrix to see that
\begin{equation}\label{K_0}{\left( {\int_{ 0}^{1} {{{\left| {\int_{ 0}^{1} {{K_0}(\gamma (t),\gamma (s))h(s)ds} } \right|}^4}} dt} \right)^{\frac{1}
{4}}} \le  C{T^{ - 1}}{\lambda ^{\frac{1}
{2}}}{\left\| h \right\|_{{L^{\frac{4}
{3}}}([ 0,1])}},\end{equation}
which is better than the bounds in (\ref{cmpmain}).
Since the kernel of $\chi (T(\lambda  + P))$ is $O(\lambda^{-N})$ with constants independent of $T$, we are left to consider the integral operator $S_\lambda$:
\begin{equation}{S_\lambda }h(t) = \frac{1}
{{\pi T}}\int_{ - \infty }^\infty  {\int_{ 0}^{1} {(1 - \beta (\tau ))\hat \chi (\tau /T){e^{i\lambda \tau }}(\cos \tau P)(\gamma (t),\gamma (s))h(s)dsd\tau.} }
\end{equation}

As in \cite{chensogge} and \cite{top}, we now use the Hadamard parametrix and the Cartan-Hadamard theorem to lift the calculations up to the universal cover $(\mathbb{R}^2,\tilde g)$ of $(M,g)$.

Let $\Gamma$ denote the group of deck transformations preserving the associated covering map $\kappa :{\mathbb{R}^2} \to M$ coming from the exponential map from $\gamma(0)$ associated with the metric $g$ on $M$. The metric $\tilde g$ is its pullback via $\kappa$. Choose also a Dirchlet fundamental domain, $D \simeq M$, for $M$ centered at the lift $\tilde \gamma(0)$ of $\gamma(0)$. We shall let $\tilde \gamma(t), t\in \mathbb{R}$ denote the lift of the geodesic $\gamma(t), t\in \mathbb{R}$, containing the unit geodesic segment $\gamma(t), t\in [0,1]$. We measure the distances in $(\mathbb{R}^2,\tilde g)$ using its Riemannian distance function $d_{\tilde g}(\ \cdot\ ,\ \emph{}\cdot\ )$.

Following \cite{chensogge}, we recall that if $\tilde x$ denotes the lift of $x\in M$ to $D$, then we have the following formula \[(\cos t\sqrt { - {\Delta _g}} )(x,y) = \sum\limits_{\alpha  \in \Gamma } {(\cos t} \sqrt { - {\Delta _{\tilde g}}} )(\tilde x,\alpha (\tilde y)).\]
Consequently, we have, for $t\in [0,1]$,
\[{S_\lambda }h(t) = \frac{1}
{{\pi T}}\sum\limits_{\alpha  \in \Gamma } {\int_{ - \infty }^\infty  {\int_{ 0}^{1} {(1 - \beta (\tau ))\hat \chi (\tau /T){e^{i\lambda \tau }}(\cos \tau \sqrt { - {\Delta _{\tilde g}}} )(\tilde \gamma (t),\alpha (\tilde \gamma (s)))h(s)\,dsd\tau }.} }
\]
Let
 \begin{equation}\label{tube}{{\rm T}_R}(\tilde \gamma ) = \{ (x,y) \in {\mathbb{R}^2}:{d_{\tilde g}}((x,y),\tilde \gamma ) \le R\}\end{equation}
and
\[{\Gamma _{{{\rm T}_R}(\tilde \gamma )}} = \{ \alpha  \in \Gamma :\alpha (D) \cap {{\rm T}_R}(\tilde \gamma ) \ne \emptyset \}. \]
From now on we fix $R\approx {\rm Inj} M$.

We write
\[{S_\lambda }h(t) = S_\lambda ^{tube}h(t) + {S_\lambda^{osc} }h(t) = \sum\limits_{\alpha  \in {\Gamma _{{{\rm T}_R}(\tilde \gamma )}}}  +  \sum\limits_{\alpha  \notin {\Gamma _{{{\rm T}_R}(\tilde \gamma )}}},t\in[0,1]. \]
By the Huygens principle,
\[(\cos \tau \sqrt { - {\Delta _{\tilde g}}} )(\tilde \gamma (t),\alpha (\tilde \gamma (s)))=0,\quad {\rm if} \quad d_{\tilde g}(\tilde \gamma (t),\alpha (\tilde \gamma (s)))>\tau.\]
Recall that $\hat \chi(\tau)=0$ if $|\tau|\ge1$. Hence ${d_{\tilde g}}(\tilde \gamma (t),\alpha (\tilde \gamma (s)))\le T, s,t\in[0,1].$

Since there are only $O(1)$ ``translates''  of $D$, $\alpha(D)$, that intersect any geodesic ball with arbitrary center of radius $R$, it follows that
\begin{equation}\label{tubenum}\# \{ \alpha  \in {\Gamma _{{{\rm T}_R}(\tilde \gamma )}}:{d_{\tilde g}}(0,\alpha (0)) \in [{2^k},{2^{k + 1}}]\}  \le C{2^k}.\end{equation}
Thus the number of nonzero summands in $S_\lambda ^{tube}h(t)$ is $O(T)$ and in $S_\lambda ^{osc}h(t)$ is $O(e^{CT})$.

Given $\alpha \in \Gamma$ set with $s,t\in[0,1]$
\[K_\alpha(t,s) = \frac{1}
{{\pi T}}\int_{ - T}^T {(1 - \beta (\tau ))\hat \chi (\tau /T){e^{i\lambda \tau }}(\cos \tau \sqrt { - {\Delta _{\tilde g}}} )(\tilde \gamma (t),\alpha (\tilde \gamma (s)))h(s)d\tau }.\]
When $\alpha=Identity$, by using the Hadamard parametrix, we get
\[|K_{\rm Id}(t,s)| \le CT^{-1}\lambda^{\frac12}|t-s|^{-\frac12}.\]
Thus, \[\left| {\int_{ 0}^{1} {K_{\rm Id}(t,s)ds} } \right| \le C{T^{ - 1}}{\lambda ^{\frac{1}
{2}}}.\]
If $\alpha \ne Identity$, we set
\[\phi (t,s) = {d_{\tilde g}}(\tilde \gamma (t),\alpha (\tilde \gamma (s))), s,t\in[0,1].\]
Then  by the Huygens principle and $\alpha \ne Identity$, we have \begin{equation}\label{phileT}2\le \phi(t,s)\le T, \quad {\rm if} \quad s,t\in[0,1].\end{equation}
Following Lemma 3.1 in \cite{chensogge}, we can write \[K_\alpha(t,s) = w(\tilde \gamma (t),\alpha (\tilde \gamma (s)))\sum\limits_ \pm  {{a_ \pm }(T,\lambda ;\phi (t,s)){e^{ \pm i\lambda \phi (t,s)}} + R(t,s)}, \]
where $|w(x,y)|\le C$, and for each $j=0,1,2,...$, there is a constant $C_j$ independent of $T$, $\lambda \ge 1$ so that \begin{equation}\label{amplitude}\left| {\partial _r^j{a_ \pm }(T,\lambda ;r)} \right| \le {C_j}{T^{ - 1}}{\lambda ^{\frac{1}
{2}}}{r^{ - \frac{1}
{2} - j}}, r\ge1.\end{equation}
Using the Hadamard parametix with an estimate on the remainder term (see \cite{hangzhou}), we see that \[|R(t,s)|\le e^{CT}.\]

Therefore we are able to estimate $S_\lambda^{tube}h$ by Young's inequality. Indeed, the kernel $K_\lambda ^{tube}(t,s)$ of $S_\lambda^{tube}$ satisfies
\[\left| {K_\lambda ^{tube}(t,s)} \right| \le C{T^{ - 1}}{\lambda ^{\frac{1}
{2}}}\sum\limits_{1\le 2^k\le T} {{2^k2^{-k/2}}}  + {e^{CT}} = C{T^{ - \frac{1}
{2}}}{\lambda ^{\frac{1}
{2}}} + {e^{CT}}.\]
Consequently,
\begin{equation}\label{stube}{\left\| {S_\lambda ^{tube}h} \right\|_{{L^4}[0,1]}} \le (C{T^{ - \frac{1}
{2}}}{\lambda ^{\frac{1}
{2}}} + {e^{CT}}){\left\| h \right\|_{{L^{\frac{4}
{3}}}[ 0,1]}}.\end{equation}

\subsection{A stationary phase argument}
To deal with the remaining part $S_\lambda^{osc}h(t)$, we need the following detailed version of the oscillatory integral estimates. (See e.g. Chapter 1 of \cite{fio}).

\begin{proposition}\label{oscint}Let $a\in C_0^\infty(\mathbb{R}^2)$, let $\phi\in C^\infty(\mathbb{R}^2)$ be real valued and $\lambda>0$, set
\[{T_\lambda }f(t) = \int_{ - \infty }^\infty  {{e^{i\lambda \phi (t,s)}}a} (t,s)f(s)ds,f\in C_0^\infty(\mathbb{R}).\]
If $\phi_{st}^{''}\ne 0$ on ${\rm supp}$ a, then \[{\left\| {{T_\lambda }f} \right\|_{{L^2}(\mathbb{R})}} \le {C_{a,\phi }}\lambda^{-\frac12}{\left\| f \right\|_{{L^2}(\mathbb{R})}},\]
where \begin{equation}\label{const1}C_{a,\phi}=C{\rm diam}({\rm supp}\ a)^{\frac12}\Bigg\{{\|a\|_\infty+\frac{ {\sum\limits_{0 \le i,j \le 2} {\|\partial _t^ia\|_\infty\|\partial _t^j\phi _{st}^{''}\|_\infty} }}
{{{{\inf |\phi _{st}^{''}|^2}}}}}\Bigg\}.\end{equation}
If $\phi_{st}^{''}(t_0,s)=0$, $\phi_{stt}^{'''}(t_0,s)\ne0$ for all $s\in{\rm supp}\,a(t_0,{}\cdot{})$, and $\phi_{st}^{''}(t,s)\ne0$ whenever $(t,s)\in {\rm supp}\,a\setminus \{(t,s):t=t_0\}$, then
\[{\left\| {{T_\lambda }f} \right\|_{{L^2}(\mathbb{R})}} \le {C_{a,\phi }'}\lambda^{-\frac14}{\left\| f \right\|_{{L^2}(\mathbb{R})}},\]
where
\begin{equation}\label{const2}C_{a,\phi}'=C{\rm diam}({\rm supp}\ a)^{\frac14}\Bigg\{\|a\|_\infty+\frac{\sum\limits_{0 \le i,j \le 2} {\|\partial _t^ia\|_\infty \|\partial _t^j\phi _{st}^{''}\|_\infty}}
{{\inf {|\phi _{st}^{''}/(t - {t_0})|^2}}}\Bigg\}.
\end{equation}
The norm $\|\cdot\|_\infty$ and the infimum are taken on ${\rm supp}$ a. The constant $C>0$ is independent of $\lambda$, $a$ and $\phi$.
\end{proposition}

\begin{proof}By a $TT^*$ argument and Young's inequality, it suffices to estimate the kernel of $T_\lambda T_\lambda^*$
\[K(s,s')=\int {{e^{i\lambda (\phi (t,s) - \phi (t,s'))}}a(t,s)\overline {a(t,s')} dt}. \]
Let \[\varphi (t,s,s') = \frac{{\phi (t,s) - \phi (t,s')}}
{{s - s'}},s\ne s',\ {\rm and}\ \varphi (t,s,s)=\phi'_s(t,s),\]
and let
\[\tilde a(t,s,s') = a(t,s)\overline {a(t,s')}. \]
Then the kernel becomes
\begin{equation}\label{kernelK}K(s,s')=\int {{e^{i\lambda (s-s')\varphi(t,s,s')}}\tilde a(t,s,s')} dt. \end{equation}

If $\phi_{st}^{''}\ne 0$ on supp $a$, then by the mean value theorem, \[|\varphi'_t(t,s,s')|=|\phi''_{st}(t,s'')|\ge {\rm inf}|\phi''_{st}|,\]
where $s''$ is some number between $s$ and $s'$.
If $\lambda (s-s')\le 1$, it is easy to see that \[|K(s,s')| \le \int {|a(t,s)||a(t,s')|dt} \le {\rm diam}({\rm supp}\ a)\|a\|_\infty^2.\]For $\lambda (s-s')\ge 1$, we integrate by parts twice to see that
\[\begin{gathered}
  |K(s,s')| \le {(\lambda |s - s'|)^{ - 2}}\int {\left| {\frac{\partial }
{{\partial t}}\left( {\frac{1}
{{\varphi'_t}}\frac{\partial }
{{\partial t}}\left( {\frac{{\tilde a}}
{{\varphi'_t}}} \right)} \right)} \right|dt}  \hfill \\
   \le C{(\lambda |s - s'|)^{ - 2}}{\rm diam}({\rm supp}\ a)\frac{{{{ \Big( {\sum\limits_{0 \le i,j \le 2} {||\partial _t^i a|{|_\infty }||\partial _t^j\phi _{st}^{''}|{|_\infty }} } \Big)}^2}}}{{\inf |\phi _{st}^{''}{|^4}}},
\end{gathered}\]
where $C$ is some constant independent of $\lambda$, $a$ and $\phi$.

Hence \[|K(s,s')|\le C{\rm diam}({\rm supp}\ a)\Bigg\{\|a\|_\infty^2+\frac{{{{\Big( {\sum\limits_{0 \le i,j \le 2} {||\partial _t^i a|{|_\infty }||\partial _t^j\phi _{st}^{''}|{|_\infty }} } \Big)}^2}}}
{{\inf |\phi _{st}^{''}{|^4}}}\Bigg\}(1+\lambda |s-s'|)^{-2},\]
again $C$ is some constant independent of $\lambda$, $a$ and $\phi$.

Consequently, \[\int {|K(s,s')|ds}\le C_{a,\phi}^2\lambda^{-1},\]
which finishes the proof of the first case.

Now we prove the second part of our proposition. Assume that $\phi_{st}^{''}(t_0,s)=0$, $\phi_{stt}^{'''}(t_0,s)\ne0$ when $s\in$ supp $a(t_0,{}\cdot{})$, and $\phi_{st}^{''}(t,s)\ne0$ whenever $(t,s)\in {\rm supp}\,a\setminus \{(t,s):t=t_0\}$. We need to use the method of stationary phase.

Let $\delta >0$. Choose $\rho \in C_0^\infty(\mathbb{R})$ satisfying $\rho(t)=1, |t|\le 1$, and $\rho (t)=0, |t|\ge 2.$
Then \[\Big|\int {{e^{i\lambda (s - s')\varphi }}} \tilde a\rho ((t-t_0)/\delta )dt\Big| \le 4\delta \| a \|_\infty ^2.\]
For the remainder term with factor $1-\rho$, we integrate by parts twice to see that if $s\neq s'$,
\[\begin{gathered}
  \Big|\int {{e^{i\lambda (s - s')\varphi }}\tilde a(1 - \rho ((t - {t_0})/\delta ))dt} \Big| \hfill \\
   \le {(\lambda |s - s'|)^{ - 2}}\int\limits_{|t - {t_0}| > \delta } {\left| {\frac{\partial }
{{\partial t}}\left( {\frac{1}
{{\varphi'_t}}\frac{\partial }
{{\partial t}}\left( {\frac{{\tilde a(1 - \rho ((t - {t_0})/\delta ))}}
{{\varphi'_t}}} \right)} \right)} \right|dt}  \hfill \\
   \le C{(\lambda |s - s'|)^{ - 2}}\frac{{{{\Big( {\sum\limits_{0 \le i,j \le 2} {||\partial _t^ia|{|_\infty }||\partial _t^i\phi _{st}^{''}|{|_\infty }} } \Big)}^2}}}
{{\inf {{(|\phi _{st}^{''}|/|t - {t_0}|)}^4}}}\int\limits_{|t - {t_0}| > \delta } {(|t - {t_0}{|^{ - 4}} + {\delta ^{ - 2}}|t - {t_0}{|^{ - 2}})dt}  \hfill \\
   \le C{\delta ^{ - 3}}{(\lambda |s - s'|)^{ - 2}}\frac{{{{\Big( {\sum\limits_{0 \le i,j \le 2} {||\partial _t^ia|{|_\infty }||\partial _t^i\phi _{st}^{''}|{|_\infty }} } \Big)}^2}}}
{{\inf {{(|\phi _{st}^{''}|/|t - {t_0}|)}^4}}}, \hfill \\
\end{gathered}\]
where $C$ is a constant independent of $\lambda$, $a$ and $\phi$.

By setting $\delta=(\lambda|s-s'|)^{-\frac12}$, we get
\[|K(s,s')|\le C\Bigg\{\|a\|_\infty^2+\frac{{{{\Big( {\sum\limits_{0 \le i,j \le 2} {||\partial _t^i a|{|_\infty }||\partial _t^j\phi _{st}^{''}|{|_\infty }} } \Big)}^2}}}
{{\inf (|\phi _{st}^{''}|/|t-t_0|)^4}}\Bigg\}(\lambda |s-s'|)^{-\frac12},\ \ {\rm if}\ s\neq s'.\]
Therefore, \[\int {|K(s,s')|ds}\le C_{a,\phi}^{'2}\lambda^{-\frac12},\]
which completes the proof.
\end{proof}

\subsection{Proof of Theorem \ref{thm2}}
Noting that ${\rm diam}({\rm supp}\ a_\pm)\le2$ and we have good control on the size of $a_\pm$ and its derivatives by \eqref{amplitude}, it remains to estimate the size of $\phi_{st}^{''}$ and its derivatives. On general manifolds with nonpositive curvature, it seems difficult to get desirable bounds. However, under our assumption of constant curvature, we can compute $\phi_{st}^{''}$ and its derivatives explicitly.

Without loss of generality, we may assume that $(M,g)$ is a compact 2-dimensional Riemannian manifold with constant curvature equal to $-1$. It is well known that the universal cover of any 2-dimensional manifolds with negative constant curvature $-1$ is the hyperbolic plane $\mathbb{H}^2$. We consider the Poincar\'e half-plane model
\[\mathbb{H}^2=\{(x,y)\in \mathbb{R}^2:y>0\},\]
with the metric given by
\[ds^2=y^{-2}(dx^2+dy^2).\]
Recall that the distance function for the Poincar\'e half-plane model is given by
\[{\rm dist}((x_1,y_1),(x_2,y_2))={\rm arcosh}\left( {1{\text{ + }}\frac{{{{({x_2} - {x_1})}^2} + {{({y_2} - {y_1})}^2}}}
{{2{y_1}{y_2}}}} \right),\]
where arcosh is the inverse hyperbolic cosine function\[{\rm arcosh}(x)={\rm ln}(x+\sqrt{x^2-1}),\,x\ge1.\]

Moreover, the geodesics are the straight vertical rays orthogonal to the $x$-axis and the half-circles whose origins are on the $x$-axis. Any pair of geodesics can intersect at at most one point. Without loss of generality, we may assume that $\tilde\gamma$ is the $y$-axis. There are three possibilities for the image $\alpha(\tilde\gamma)$. It can be a straight line parallel to $\tilde\gamma$, a half-circle parallel to $\tilde\gamma$, or a half-circle intersecting $\tilde\gamma$ at one point. We need to treat these cases separately.

Let $\tilde\gamma(t)=(0,e^t),\ t\in \mathbb{R},$ be the infinite geodesic parameterized by arclength. Our unit geodesic segment is given by $\tilde\gamma(t),\ t\in[0,1].$ Then its image $\alpha(\tilde\gamma(s)),\ s\in [0,1],$ is a unit geodesic segment of $\alpha(\tilde\gamma)$.
\begin{lemma}\label{lemma31}If $\alpha  \notin {\Gamma _{{{\rm T}_R}(\tilde \gamma )}}$ and  $\alpha(\tilde\gamma)\cap\tilde\gamma=\emptyset$, then we have
\[{\rm inf}\ |\phi_{st}^{''}|\ge e^{-CT},\]
and
\[\|\phi_{st}^{''}\|_\infty+\|\phi_{stt}^{'''}\|_\infty+\|\phi_{sttt}^{''''}\|_\infty\le e^{CT},\]
where  $C>0$ is independent of $T$. The infimum and the norm are taken on the unit square $\{(t,s)\in \mathbb{R}^2:t,s\in[0,1]\}$.
\end{lemma}

\begin{lemma}\label{lemma33}Let $\alpha  \notin {\Gamma _{{{\rm T}_R}(\tilde \gamma )}}$ and $\alpha(\tilde\gamma)$ is a half-circle intersecting $\tilde\gamma$ at the point $(0,e^{t_0}),\ t_0\in \mathbb{R}$. 

If $t_0\notin[-1,2]$, then the intersection point $(0,e^{t_0})$ is outside some neighbourhood of the unit geodesic segment $\{\tilde \gamma(t):\,t\in[0,1]\}$. We have
\[{\rm inf}\ |\phi_{st}^{''}|\ge e^{-CT},\]
and
\[\|\phi_{st}^{''}\|_\infty+\|\phi_{stt}^{'''}\|_\infty+\|\phi_{sttt}^{''''}\|_\infty\le e^{CT},\]
where  $C>0$ is independent of $T$.

On the other hand, if $t_0\in[0,1]$, we have
\[{\rm inf}\ |\phi_{st}^{''}/(t-t_0)|\ge e^{-CT},\]
and
\[\|\phi_{st}^{''}\|_\infty+\|\phi_{stt}^{'''}\|_\infty+\|\phi_{sttt}^{''''}\|_\infty\le e^{CT},\]
where  $C>0$ is independent of $T$. The infima and the norms are taken on the unit square $\{(t,s)\in \mathbb{R}^2:t,s\in[0,1]\}$.
\end{lemma}

We shall postpone the proof of Lemma \ref{lemma31} and Lemma \ref{lemma33} to the last section. Now we see first how to finish the proof of Theorem \ref{thm2} using Lemma \ref{lemma31} and Lemma \ref{lemma33}.
\begin{proof}[Proof of Theorem \ref{thm2}]
By \eqref{K_0} and \eqref{stube}, we only need to show that
\begin{equation}\label{sosc}\|S_\lambda^{osc}h\|_{L^4([0,1])}\le e^{CT}\lambda^{\frac38}\|h\|_{L^{\frac43}([0,1])},\end{equation}
where $C$ is independent of $T$.

Recall that the number of nonzero summands in $S_\lambda^{osc}$ is $O(e^{CT})$.

Let $\alpha \notin \Gamma_{{\rm T}_R}(\tilde \gamma)$. If $\alpha(\tilde\gamma)\cap\tilde\gamma=\emptyset$, by Proposition \ref{oscint}, Lemma \ref{lemma31} and the condition on the amplitude (\ref{amplitude}), we have
\[\|S_\lambda^{osc}h\|_{L^2([0,1])}\le e^{CT}\|h\|_{L^2([0,1])}.\]

Assume that $\alpha(\tilde\gamma)$ intersects $\tilde\gamma$ at the point $\tilde\gamma(t_0)$. Since $\alpha  \notin {\Gamma _{{{\rm T}_R}(\tilde \gamma )}}$, the intersection point cannot lie on the unit geodesic segment $\alpha(\tilde \gamma(s)),\ s\in[0,1]$. Thus, by Proposition \ref{oscint}, Lemma \ref{lemma33} and (\ref{amplitude}) we obtain
\[\|S_\lambda^{osc}h\|_{L^2([0,1])}\le e^{CT}\|h\|_{L^2([0,1])},\ {\rm if}\ t_0\notin[-1,2],\]
and
\[\|S_\lambda^{osc}h\|_{L^2([0,1])}\le e^{CT}\lambda^{\frac14}\|h\|_{L^2([0,1])},\ {\rm if}\ t_0\in[0,1].\]

If $t_0\in (1,2]$, we can extend the interval $[0,1]$ to $[0,2]$. Since
\[\Big|\int_{ 0}^{1}\Big|\le \Big|\int_{ 0}^{2}\Big|+\Big|\int_{ 1}^{2}\Big|,\]
we can see from (\ref{kernelK}) that it is reduced to the second case of Proposition \ref{oscint}  and Lemma \ref{lemma33}. It is similar for $t_0\in [-1,0)$. Thus we have
\[\|S_\lambda^{osc}h\|_{L^2([0,1])}\le e^{CT}\lambda^{\frac14}\|h\|_{L^2([0,1])},\ {\rm if}\ t_0\in[-1,0)\cup (1,2].\]

Consequently, for $\lambda >1$ we always have
\[\|S_\lambda^{osc}h\|_{L^2([0,1])}\le e^{CT}\lambda^{\frac14}\|h\|_{L^2([0,1])}.\]
By interpolating with the trivial $L^1 \to L^\infty$ bound, we obtain (\ref{sosc}), finishing the proof.
\end{proof}
\subsection{Proof of Lemmas}
Before proving the lemmas, we remark that in the Poincar\'e half-plane model
\[{{\rm T}_R}(\tilde\gamma)=\{(x,y)\in \mathbb{R}^2: y>0\ {\rm and}\ y\ge |x|/\sqrt{({\rm cosh}R)^2-1}\}.\]
Indeed, the distance between $(0,e^t)$ and $(x,y)$, $y>0$, is \[f(t)={\rm arcosh}\Big(1+\frac{x^2+(y-e^t)^2}{2ye^t}\Big)={\rm arcosh}\Big(\frac{x^2+y^2+e^{2t}}{2ye^t}\Big).\]
Setting $f'(t)=0$ gives $t={\rm ln}\sqrt{x^2+y^2}$, which must be the only minimum point. Thus the distance between $(x,y)$ and the infinite geodesic $\tilde\gamma$ is \[{\rm dist}((x,y),\tilde\gamma)={\rm arcosh}(\sqrt{1+(x/y)^2}).\]
Since ${\rm dist}((x,y),\tilde\gamma)\le R$ in ${{\rm T}_R}(\tilde\gamma)$, it follows that $y\ge |x|/\sqrt{({\rm cosh}R)^2-1}$.

From now on, we shall always parametrize $\tilde\gamma$ and $\alpha(\tilde\gamma)$ by arc-length, denoted by $\gamma_1(t)$ and $\gamma_2(s)$ respectively. The explicit expressions for the corresponding segments that we concern will be given in the proof case by case.
\begin{proof}[Proof of Lemma \ref{lemma31}]
Note that in this case, the image $\alpha(\tilde\gamma)=\gamma_2$ can be either a straight line or a half-circle parallel to $\tilde\gamma=\gamma_1$. We treat these two cases separately.

\begin{figure}
  \centering
    \includegraphics[height=8cm]{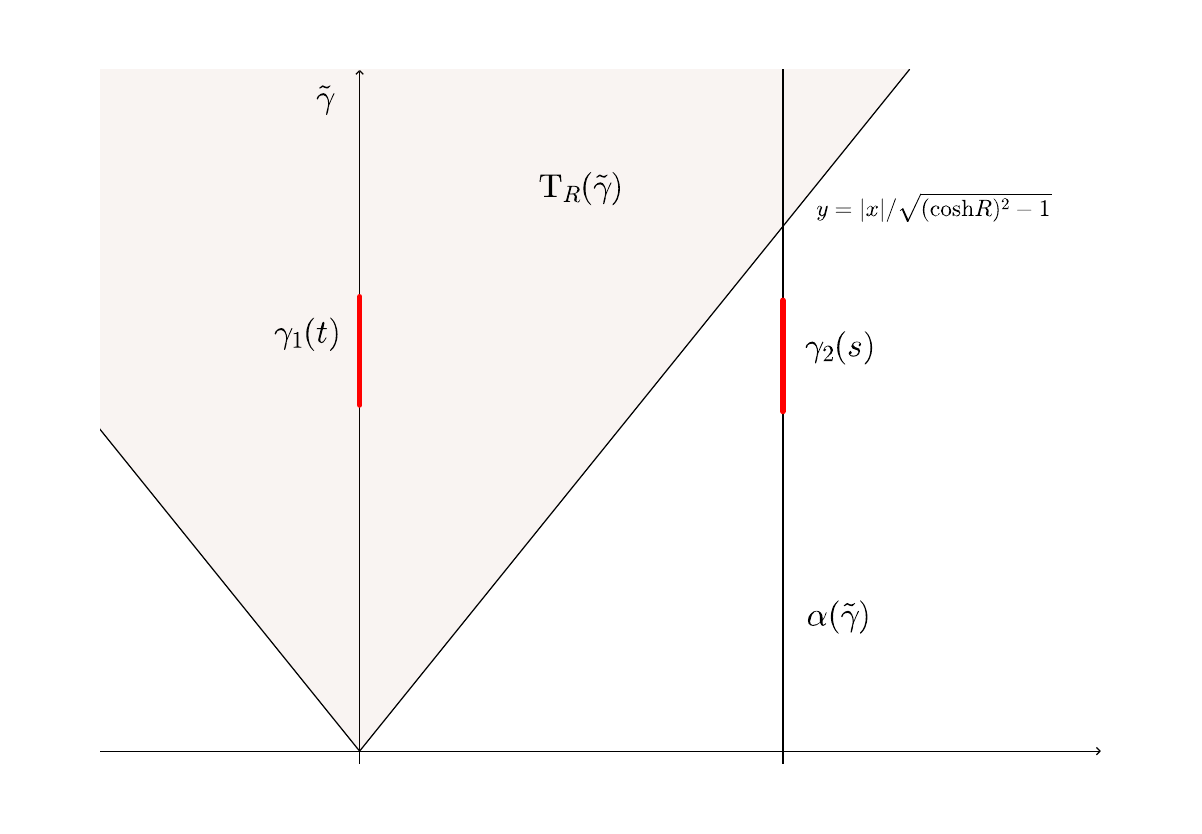}
\caption{$\alpha(\tilde\gamma)$ is a line parallel to $\tilde\gamma$.}
  \label{1}
\end{figure}

Let $\gamma_1(t)=(0,e^t), t\in[0,1]$, $\gamma_2(s)=(a,e^s)$ be the two unit geodesic segments, where $a\in \mathbb{R}$ and $s$ is in some unit closed interval of $\mathbb{R}$. See Figure 1.

The distance function is
\[\phi(t,s)={\rm dist}(\gamma_1(t),\gamma_2(s))={\rm arcosh}\Big(1+\frac{a^2+(e^s-e^t)^2}{2e^{s+t}}\Big)={\rm arcosh}\Big(\frac{a^2+e^{2t}+e^{2s}}{2e^{t+s}}\Big).\]
Then we have
\[\phi_{st}^{''}(t,s)=\frac{-8e^{2s+2t}a^2}{((a^2+e^{2t}+e^{2s})^2-4e^{2s+2t})^{3/2}}.\]

By (\ref{phileT}), we have $\phi \le T$. Thus \[a^2e^{-t-s}+e^{t-s}+e^{s-t}\le 2{\rm cosh}T,\]which gives $s\in [-T,T+1]$ and $|a|\le Ce^T$. Here $C$ is independent of $T$.

To get the lower bound of $|\phi_{st}^{''}|$, we need to use the condition that $\alpha  \notin {\Gamma _{{{\rm T}_R}(\tilde \gamma )}}$.

We claim that \begin{equation}\alpha  \notin {\Gamma _{{{\rm T}_R}(\tilde \gamma )}}\Rightarrow \ |a|\ge Ce^{-T},\end{equation} where $C$ is independent of $T$. Note that if the segment $\gamma_2(s), s\in [-T,T+1]$ is completely included in ${{\rm T}_R}(\tilde \gamma )$, then we must have $\alpha  \in {\Gamma _{{{\rm T}_R}(\tilde \gamma )}}$, meaning that \[e^{-T}\ge |a|\sqrt{({\rm cosh}R)^2-1}\ \Rightarrow \ \alpha  \in {\Gamma _{{{\rm T}_R}(\tilde \gamma )}},\] which implies our claim.

Consequently,
\[|\phi_{st}^{''}|\ge C\frac{e^{-2T}e^{-2T}}{e^{6T}}=Ce^{-10T}.\]
This gives the lower bound of $|\phi_{st}^{''}|$.

The upper bounds can be estimated similarly. Note that
\[(a^2+e^{2t}+e^{2s})^2-4e^{2s+2t}\ge a^2 \ge Ce^{-2T}.\]

We have
\[|\phi_{st}^{''}|\le C\frac{e^{2T}e^{2T}}{e^{-{3T}}}\le Ce^{7T}.\]

Moreover, \[\phi_{stt}^{'''}=\frac{-16a^2e^{2s+2t}((a^2+e^{2s})^2+e^{2s+2t}-a^2e^{2t}-2e^{4t})}{((a^2+e^{2t}+e^{2s})^2-4e^{2s+2t})^{5/2}},\]
and
\[\phi_{sttt}^{''''}=\frac{-32a^2 e^{2s+2t}((a^2+e^{2s})^4+{\rm lower\ order\   terms})}{((a^2+e^{2t}+e^{2s})^2-4e^{2s+2t})^{7/2}}.\]
 Thus
\[|\phi_{stt}^{'''}|\le C\frac{e^{2T}e^{2T}e^{4T}}{e^{-5T}}\le Ce^{13T},\]
and \[|\phi_{sttt}^{''''}|\le C\frac{e^{2T}e^{2T}e^{8T}}{e^{-7T}}\le Ce^{19T}.\]
This completes the proof of the first case.

Now we turn to the case when $\gamma_2$ is a half-circle centered at $(a,0)$ with radius $r>0$. See Figure 2. Let $\gamma_1(t)=(0,e^t), t\in[0,1]$, $\gamma_2(s)=(a+r\frac{1-e^{2s}}{1+e^{2s}},\frac{2re^s}{1+e^{2s}})$ be two unit geodesic segments, where $|a|\ge r>0$ and $s$ is in some unit closed interval of $\mathbb{R}$. Without loss of generality, we may only consider the case $a\ge r>0$. Then the distance function is
\begin{equation}\label{phicircle}\phi(t,s)={\rm dist}(\gamma_1(t),\gamma_2(s))= {\rm arcosh}\Big( \frac{A}{4re^{s+t}}\Big),\end{equation}
where \begin{equation}\label{capitala}A=e^{2s+2t}+(a-r)^2e^{2s}+e^{2t}+(a+r)^2.\end{equation}

\begin{figure}
  \centering
    \includegraphics[height=8cm]{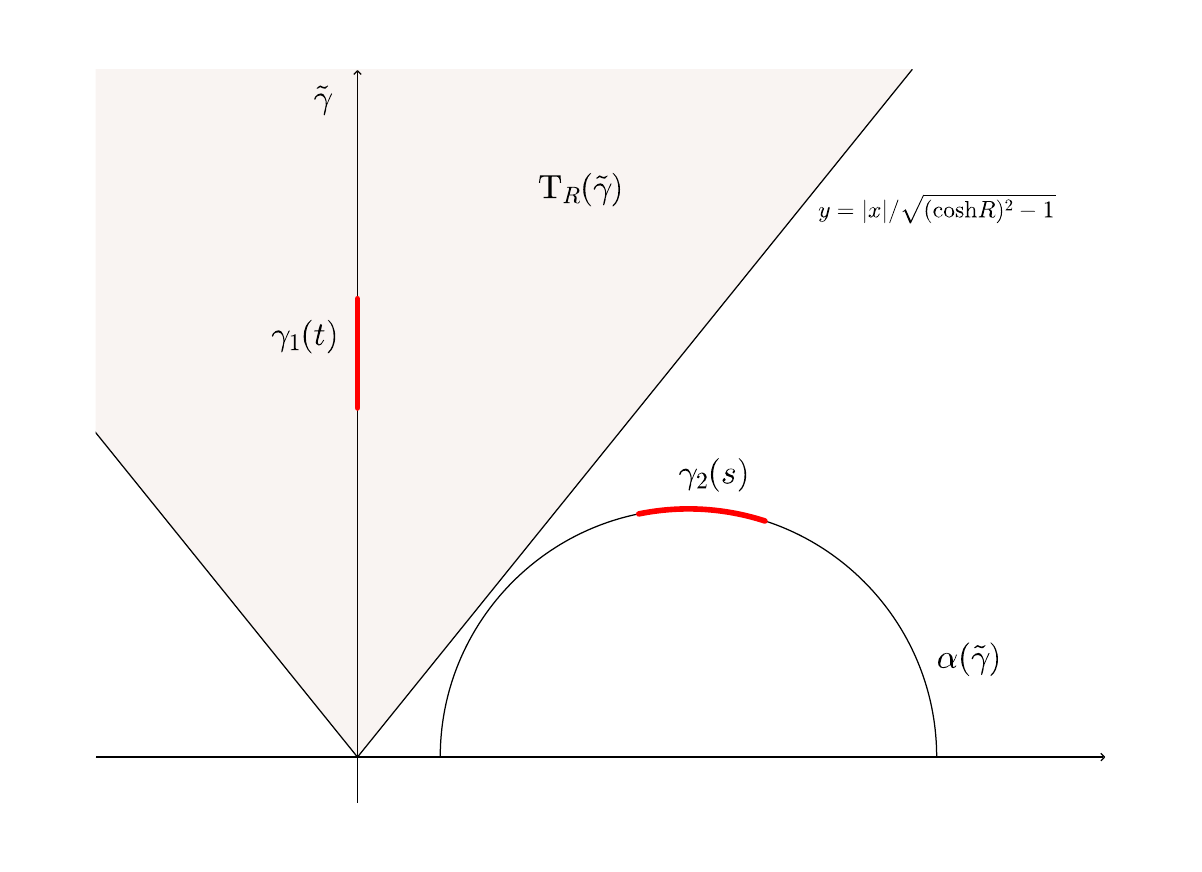}
\caption{$\alpha(\tilde\gamma)$ is a half-circle parallel to $\tilde\gamma$.}
  \label{2}
\end{figure}

Thus we have
\begin{equation} \label{phist} \phi_{st}^{''}=\frac{16re^{2s+2t}(a+r+(a-r)e^{2s})(a^2-r^2+e^{2t})}{(A^2-16r^2e^{2s+2t})^{3/2}}.\end{equation}

Again by (\ref{phileT}), we get $\phi\le T$. Namely,
\begin{equation}\label{poly}(e^{2t}+(a-r)^2)e^{2s}-4r({\rm cosh}T)e^te^s+e^{2t}+(a+r)^2\le 0,\end{equation}
which implies
\begin{equation}\label{segment2}\frac{r}{4{\rm cosh}T}\le e^s\le 4r{\rm cosh}T.\end{equation}
Moreover, note that if we view the left hand side of \eqref{poly} as a quadratic polynomial in terms of $e^s$, then the discriminant has to be nonnegative:
\[16r^2({\rm cosh}T)^2e^{2t}-4(e^{2t}+(a-r)^2)(e^{2t}+(a+r)^2)\ge0,\]
we obtain that
\begin{equation}\label{a/rbound}\frac{a}{r}\le 2e{\rm cosh}T,\end{equation}
and
\begin{equation}\label{a-rbound}|a-r|\le 2e {\rm cosh}T.\end{equation}

To get the lower bound of $|\phi_{st}^{''}|$, we need to use the condition that $\alpha  \notin {\Gamma _{{{\rm T}_R}(\tilde \gamma )}}$.

We claim that there exists some constant $C$ independent of $T$ such that
\begin{equation}\label{claim32} \alpha  \notin {\Gamma _{{{\rm T}_R}(\tilde \gamma )}} \Rightarrow r\le C{\rm cosh}T\ {\rm or}\ |a-r|\ge \frac1{C{\rm cosh}T}.\end{equation}

In fact, if the segment $\gamma_2(s), s\in [-{\rm ln}(4r^{-1}{\rm cosh}T),{\rm ln}(4r{\rm cosh}T)]$ is completely included in ${{\rm T}_R}(\tilde \gamma )$, then we must have $\alpha  \in {\Gamma _{{{\rm T}_R}(\tilde \gamma )}}$. By some basic calculations, we can see that
\begin{equation}\label{intube}\{\gamma_2(s):s\in \mathbb{R}\} \cap {{\rm T}_R}(\tilde \gamma )=\{\gamma_2(s):(a/r-1)e^{2s}-2e^s\sqrt{({\rm cosh}R)^2-1}+a/r+1\le 0\}.\end{equation}

If $a\ne r$ and $a/r\le {\rm cosh}R$, the RHS of (\ref{intube}) becomes
\begin{equation}\label{condit1}\{\gamma_2(s):u_- \le e^s \le u_+\},\end{equation}
where \begin{equation}\label{condit2}u_\pm=\frac{\sqrt{({\rm cosh}R)^2-1}\pm \sqrt{({\rm cosh}R)^2-(a/r)^2}}{a/r-1}.\end{equation}
Note that  $u_- \le \frac{r}{4{\rm cosh}T}$ and $4r{\rm cosh}T \le u_+$ imply $\alpha  \in {\Gamma _{{{\rm T}_R}(\tilde \gamma )}}$.
 We have
\begin{equation}\label{contra}r \ge 4\sqrt{\frac{{\rm cosh}R+1}{{\rm cosh}R-1}}{\rm cosh}T\ {\rm and}\ |a-r|\le \frac{\sqrt{({\rm cosh}R)^2-1}}{4{\rm cosh}T} \Rightarrow \alpha  \in {\Gamma _{{{\rm T}_R}(\tilde \gamma )}}.\end{equation}
If $a=r$, similarly we have
\[r \ge \frac{4{\rm cosh}T}{\sqrt{({\rm cosh}R)^2-1}} \Rightarrow \alpha  \in {\Gamma _{{{\rm T}_R}(\tilde \gamma )}},\]
which finishes the proof of our claim.

Therefore, by (\ref{claim32}) we have to consider two cases and estimate $|\phi_{st}^{''}|$ respectively.
We note that by $\phi\le T$,\begin{equation}\label{phitimes}|\phi_{st}^{''}|\ge |\phi_{st}^{''}|\left( \frac{A}{4re^{s+t}{\rm cosh}T}\right)^2\ge \frac{|a+r+(a-r)e^{2s}||a^2-r^2+e^{2t}|}{({\rm cosh}T)^2rA}.\end{equation}

(I) Assume $r\le C {\rm cosh}T$.

If $a-r\ge 1$, then by (\ref{a/rbound}), we get $r\ge (2e{\rm cosh}T)^{-1}$. Thus, we obtain
\[|\phi_{st}^{''}|\ge \frac{C}{({\rm cosh}T)^2}\frac{(a-r)^2(a+r)r^2({\rm cosh}T)^{-2}}{r(r^2(a-r)^2({\rm cosh}T)^2)}\ge Ce^{-6T}.\]

If $a-r\le1$ and $r\ge ({\rm cosh}T)^{-1}$, following (\ref{phitimes}) we have
\[|\phi_{st}^{''}|\ge \frac{C}{({\rm cosh}T)^2}\frac{a+r}{r(r^2({\rm cosh}T)^2)}\ge Ce^{-6T}.\]

If $a-r\le1$ and $r\le ({\rm cosh}T)^{-1}$, again by (\ref{phitimes}) we get
\[|\phi_{st}^{''}|\ge \frac{C}{({\rm cosh}T)^2}\frac{a+r}{r}\ge Ce^{-2T}.\]

(II) Assume $a-r\ge \frac1{C {\rm cosh}T}$. We may assume $r\ge 1$, otherwise it is reduced to the first case.

If $a-r\ge 1$, then
\[|\phi_{st}^{''}|\ge \frac{C}{({\rm cosh}T)^2}\frac{(a+r)(a-r)^2r^2({\rm cosh}T)^{-2}}{r((a-r)^2r^2({\rm cosh}T)^2)}\ge Ce^{-6T}.\]

If $a-r\le 1$ then
\[|\phi_{st}^{''}|\ge \frac{C}{({\rm cosh}T)^2}\frac{(a+r)(a-r)^2r^2({\rm cosh}T)^{-2}}{r(r^2({\rm cosh}T)^2)}\ge Ce^{-8T}.\]

Note that the constant $C$ is independent of $T$. Hence we finish the proof of the lower bound of $|\phi_{st}^{''}|$.

The upper bounds can be obtained in a similar fashion.
Direct computations give
\begin{equation}\label{phistt}\phi_{stt}^{'''}=\frac{-32re^{2s+2t}(a+r+(a-r)e^{2s})((a+r)(a-r)^5e^{4s}+{\rm lower\ order\ terms})}{(A^2-16r^2e^{2s+2t})^{5/2}}\end{equation}

and
\begin{equation}\label{phisttt}\phi_{sttt}^{''''}=\frac{-64re^{2s+2t}(a+r+(a-r)e^{2s})((a+r)(a-r)^9e^{8s}+{\rm lower\ order\ terms})}{(A^2-16r^2e^{2s+2t})^{7/2}}.\end{equation}

(I) Assume $r\le C {\rm cosh}T$.
Observe that \[A-4re^{s+t}\ge \frac{(e^{2t}+a^2-r^2)^2}{e^{2t}+(a-r)^2}\ge e^{2t}\ge 1,\]
we have \[A^2-16r^2e^{2s+2t}\ge 1.\]
Then
\[|\phi_{st}^{''}|\le C({\rm cosh}T)({\rm cosh}T)^4({\rm cosh}T)^5({\rm cosh}T)^2\le Ce^{12T},\]
\[|\phi_{stt}^{'''}|\le C({\rm cosh}T)({\rm cosh}T)^4({\rm cosh}T)^5({\rm cosh}T)^{14}\le Ce^{24T},\]
and
\[|\phi_{stt}^{'''}|\le C({\rm cosh}T)({\rm cosh}T)^4({\rm cosh}T)^5({\rm cosh}T)^{26}\le Ce^{36T}.\]

(II) Assume $|a-r|\ge \frac1{C {\rm cosh}T}$. We may assume $r\ge {\rm cosh}T$, otherwise it is reduced to the first case. Note that
 \[\begin{aligned}
  {A^2} - 16{r^2}{e^{2s + 2t}} &\ge ({({e^{2t}} + {(C{\text{cosh}}T)^{ - 2}}{e^{2s}} + 4{r^2})^2} - 16{r^2}{e^{2s + 2t}}\\
   &\ge {(C{\text{cosh}}T)^{ - 2}}{e^{4s}} + {({e^{2s + 2t}} - 4{r^2})^2}. \end{aligned}\]
We have \begin{equation}\label{estimateA} A^2-16r^2e^{2s+2t}\ge C({\rm cosh}T)^{-6}r^4.\end{equation}
Thus
\[|\phi_{st}^{''}|\le \frac{Cr^3({\rm cosh}T)^2(r^2({\rm cosh}T)^3)(r{\rm cosh}T)}{(({\rm cosh}T)^{-6}r^4)^{3/2}}\le Ce^{15T},\]
\[|\phi_{stt}^{'''}|\le \frac{Cr^3({\rm cosh}T)^2(r^2({\rm cosh}T)^3)(r^5({\rm cosh}T)^9)}{(({\rm cosh}T)^{-6}r^4)^{5/2}}\le Ce^{29T},\]
and
\[|\phi_{sttt}^{''''}|\le \frac{Cr^3({\rm cosh}T)^2(r^2({\rm cosh}T)^3)(r^9({\rm cosh}T)^{17})}{(({\rm cosh}T)^{-6}r^4)^{7/2}}\le Ce^{43T}.\]
Since the constant $C$ is independent of $T$, the proof is complete.
\end{proof}

\begin{proof}[Proof of Lemma \ref{lemma33}]
Let $\gamma_1(t)=(0,e^t), t\in[0,1]$, $\gamma_2(s)=(a+r\frac{1-e^{2s}}{1+e^{2s}},\frac{2re^s}{1+e^{2s}})$ be two unit geodesic segments, where $r>|a|\ge 0$ and $s$ is in some unit closed interval of $\mathbb{R}$. Without loss of generality, we may only consider the case $r>a\ge 0$. The expressions of the distance function $\phi$ and its derivatives are the same as in (\ref{phicircle})-(\ref{phist}),(\ref{phistt}) and (\ref{phisttt}).  (\ref{segment2})-(\ref{phitimes}) also hold in this case.

The zero set of $\phi_{st}^{''}(t,s)$ is
\[\Big\{(t,s)\in \mathbb{R}^2: t=t_0\ {\rm or}\ s=s_0,\ {\rm where}\ e^{2t_0}=r^2-a^2\  {\rm and}\  e^{2s_0}=\frac{r+a}{r-a}\Big\}.\]

It is not difficult to check that the point $p=\gamma_1(t_0)=\gamma_2(s_0)$ is the intersection point of the two infinite geodesics $\gamma_1$ and $\gamma_2$. If the unit geodesic segment $\gamma_2(s)$ passes through the intersection point, then this segment must be contained in ${\rm T}_R(\tilde \gamma)$, thus our geodesic segment $\gamma_2(s)$ cannot pass through the point $p$. 

We need to consider the following two cases: (I) $t_0\in[0,1]$, (II) $t_0\notin[-1.2]$.

(I) Assume $t_0\in [0,1]$. See Figure 3.

\begin{figure}
  \centering
    \includegraphics[height=8cm]{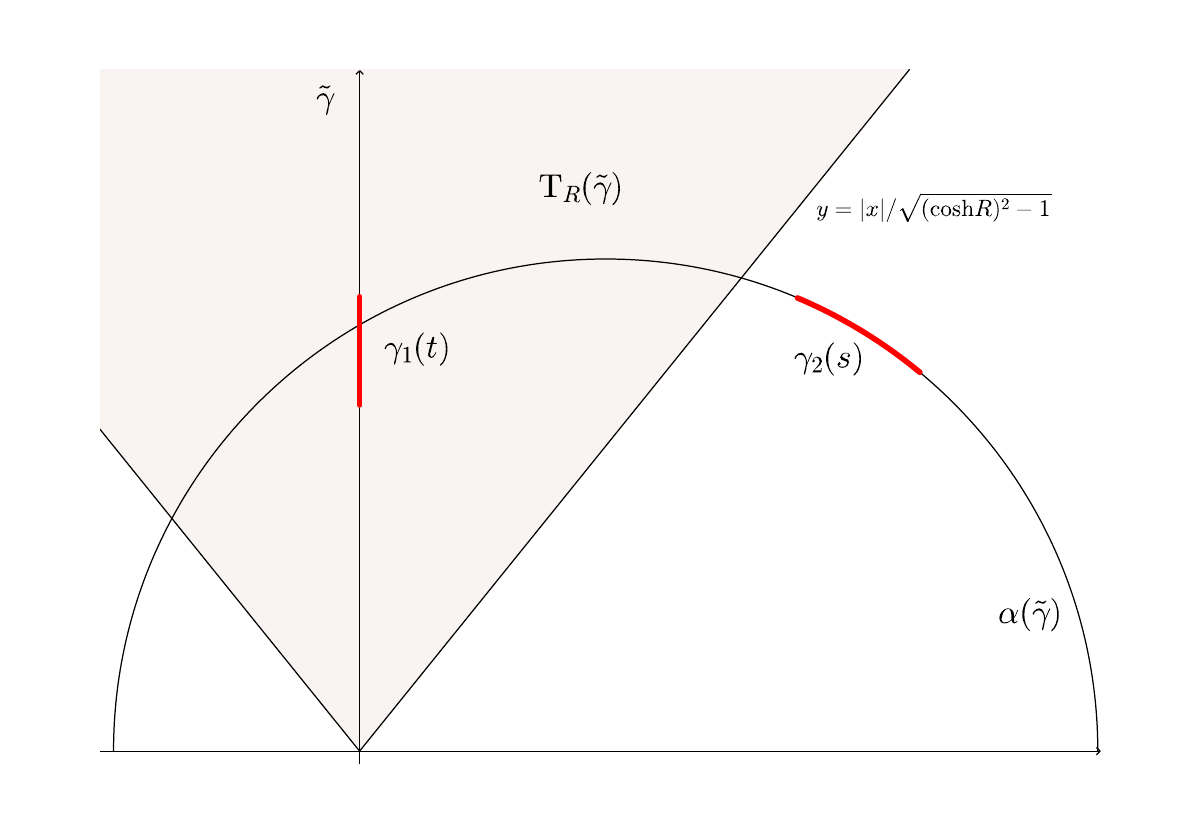}
\caption{$t_0\in[0,1]$}
  \label{3}
\end{figure}

We claim that
\begin{equation}\alpha  \notin {\Gamma _{{{\rm T}_R}(\tilde \gamma )}} \Rightarrow r\le C{\rm cosh}T, \end{equation}
where $C$ is independent of $T$.

To see this, we need to find some sufficient conditions to ensure that $\alpha  \in {\Gamma _{{{\rm T}_R}(\tilde \gamma )}}$. We note that (\ref{condit1})-(\ref{contra}) still hold, as $a/r <1$. Moreover, since $t_0\in [0,1]$, we have
\begin{equation}\label{t0}
e^{2t_0}=(r+a)(r-a)\in [1,e^2],\end{equation}
which implies that the two conditions in (\ref{claim32}), $ r\le C{\rm cosh}T$ and $|r-a|\ge \frac1{C{\rm cosh}T}$, are equivalent. Hence our claim is valid.

By (\ref{condit1}) and (\ref{condit2}), $\alpha  \notin {\Gamma _{{{\rm T}_R}(\tilde \gamma )}}$ implies $e^s> u_+$ or $e^s<u_-$.

If $e^s> u_+$, direct calculations yield
\[(r-a)e^{2s}-(r+a)\ge \frac{2(({\rm cosh}R)^2-a/r)(r-a)}{(1-a/r)^2}-2r\ge 2(({\rm cosh}R)^2-1)r.\]

If $e^s<u_-$, similarly we have
\[(r+a)-(r-a)e^{2s}\ge 2r-\frac{({\rm cosh}R)^2(r-a)}{({\rm cosh}R)^2-1}\ge \frac{({\rm cosh}R)^2-2}{({\rm cosh}R)^2-1}r.\]

Hence for some constant $C$ independent of $T$, we always have
\begin{equation}\label{middleterm}|a+r+(a-r)e^{2s}|\ge Cr.\end{equation}
Note that $1\le r\le C{\rm cosh}T$ and $|r-a|\le e$. We get
\[A\le Cr^2({\rm cosh}T)^2.\]
Thus by (\ref{phitimes})
\[\left|\frac{\phi_{st}^{''}}{t-t_0}\right|\ge \frac{Cr}{({\rm cosh}T)^2 r(r^2({\rm cosh}T)^2)}\left|\frac{e^{2t}-e^{2t_0}}{t-t_0}\right|\ge Ce^{-6T}.\]

Now we prove the upper bounds. Since $|r-a|\ge \frac1{C{\rm cosh}T}$, (\ref{estimateA}) still holds in this case. Thus by (\ref{phist}), (\ref{phistt}) and (\ref{phisttt}), we see that
\[|\phi_{st}^{''}|\le \frac{Cr^3({\rm cosh}T)^2(r^2({\rm cosh}T)^2)(r)}{(({\rm cosh}T)^{-6}r^4)^{3/2}}\le Ce^{13T},\]

\[|\phi_{stt}^{'''}|\le \frac{Cr^3({\rm cosh}T)^2(r^2({\rm cosh}T)^2)(r^5({\rm cosh}T)^4)}{(({\rm cosh}T)^{-6}r^4)^{5/2}}\le Ce^{23T},\]

and
\[|\phi_{sttt}^{''''}|\le \frac{Cr^3({\rm cosh}T)^2(r^2({\rm cosh}T)^2)(r^9({\rm cosh}T)^8)}{(({\rm cosh}T)^{-6}r^4)^{7/2}}\le Ce^{33T}.\]

(II) Assume $t_0\notin[-1,2]$. See Figure 4.

\begin{figure}
  \centering
    \includegraphics[height=8cm]{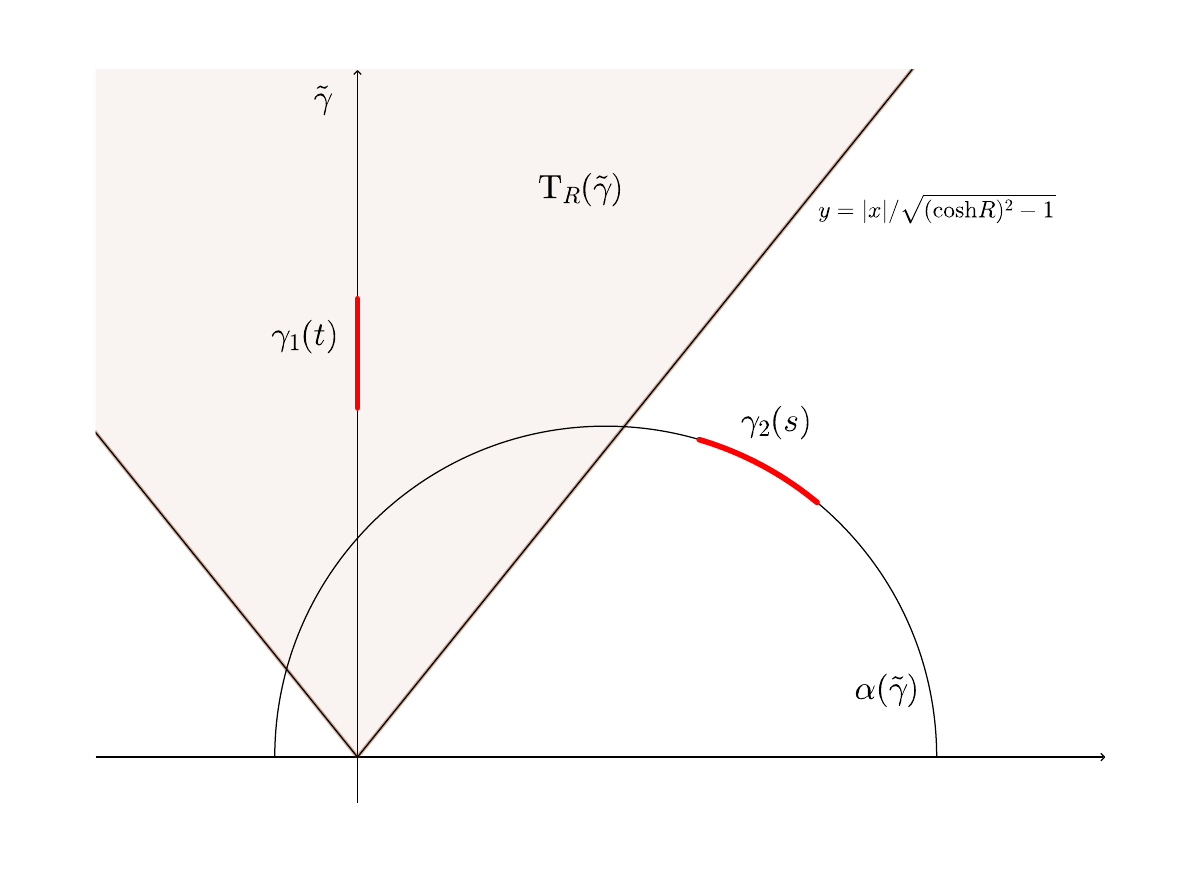}
\caption{$t_0\not\in[-1,2]$}
  \label{4}
\end{figure}

Note we also have the claim (\ref{claim32}), as $a/r<1$. By (\ref{a-rbound}), we have
\[A\le Cr^2({\rm cosh}T)^4.\]

If $r\le C({\rm cosh}T)^4$, by (\ref{phitimes}), (\ref{middleterm}) and $t_0\notin[-1,2]$, we have
\[|\phi_{st}^{''}|\ge \frac{Cr}{({\rm cosh}T)^2r(r^2({\rm cosh}T)^4)}\ge Ce^{-14T}.\]

If $|r-a|\ge \frac1{C{\rm cosh}T}$ and $r\ge ({\rm cosh}T)^4$, then
\[(r-a)e^{2s}-(a+r)\ge Cr^2({\rm cosh}T)^{-3}-2r\ge Cr^2({\rm cosh}T)^{-3},\]and
\[r^2-a^2-e^{2t}\ge Cr({\rm cosh}T)^{-1}-e^2\ge Cr({\rm cosh}T)^{-1}.\]

Thus by (\ref{phitimes}) we get
\[|\phi_{st}^{''}|\ge \frac{C(r^2({\rm cosh}T)^{-3})(r({\rm cosh}T)^{-1})}{({\rm cosh}T)^2r(r^2({\rm cosh}T)^4)}\ge Ce^{-10T}.\]

It remains to prove the upper bounds. Since $t_0\notin[-1,2]$ and (\ref{a-rbound}), we have
\[A-4re^{s+t}\ge \frac{(e^{2t}+a^2-r^2)^2}{e^{2t}+(a-r)^2}\ge C({\rm cosh}T)^{-2}.\]
If $r\le C{\rm cosh}T$, then
\[|\phi_{st}^{''}|\le \frac{C({\rm cosh}T)({\rm cosh}T)^4({\rm cosh}T)^5({\rm cosh}T)^2}{(({\rm cosh}T)^{-2})^{3/2}}\le Ce^{15T},\]
\[|\phi_{stt}^{'''}|\le \frac{C({\rm cosh}T)({\rm cosh}T)^4({\rm cosh}T)^5({\rm cosh}T)^{14}}{(({\rm cosh}T)^{-2})^{5/2}}\le Ce^{29T},\]
and
\[|\phi_{sttt}^{''''}|\le \frac{C({\rm cosh}T)({\rm cosh}T)^4({\rm cosh}T)^5({\rm cosh}T)^{26}}{(({\rm cosh}T)^{-2})^{7/2}}\le Ce^{43T}.\]

If $|r-a|\ge \frac1{C{\rm cosh}T}$ and $r\ge {\rm cosh}T$ then we also have (\ref{estimateA}).
Hence the same estimates hold:
\[|\phi_{st}^{''}|\le \frac{Cr^3({\rm cosh}T)^2(r^2({\rm cosh}T)^3)(r{\rm cosh}T)}{(({\rm cosh}T)^{-6}r^4)^{3/2}}\le Ce^{15T},\]
\[|\phi_{stt}^{'''}|\le \frac{Cr^3({\rm cosh}T)^2(r^2({\rm cosh}T)^3)(r^5({\rm cosh}T)^9)}{(({\rm cosh}T)^{-6}r^4)^{5/2}}\le Ce^{29T},\]
and
\[|\phi_{sttt}^{''''}|\le \frac{Cr^3({\rm cosh}T)^2(r^2({\rm cosh}T)^3)(r^9({\rm cosh}T)^{17})}{(({\rm cosh}T)^{-6}r^4)^{7/2}}\le Ce^{43T},\]
which completes our proof.
\end{proof}
\begin {remark} \label{remark2}{\rm As pointed out in \cite{top}, the various upper bounds $|D^\alpha \phi|\le C_\alpha e^{CT}$ also follow from Proposition 3 and Lemma 4 in \cite{berard}. We are including the proofs for these upper bounds in our case just for the sake of completeness.}
\end{remark}

\section{Acknowledgement}
We would like to thank Professor C. Sogge for his guidance and patient discussions during this study. This paper would not have been possible without his generous support. It's our pleasure to thank our colleagues C. Antelope, D. Ginsberg, and X. Wang  for going through an early draft of this paper.

\bibliography{ER}

\bibliographystyle{plain}

\end{document}